\documentclass[12pt,twoside,reqno]{amsart} 
\usepackage{amsmath, amsfonts, amsthm, amssymb}
\usepackage[english]{babel}
\usepackage[latin1]{inputenc}
\usepackage{amsmath}
\usepackage{amsfonts}
\usepackage{amsthm}
\usepackage[pdftex]{graphicx}
\usepackage{color}
\usepackage{alltt}
\usepackage{enumitem}

\usepackage{caption}
\usepackage{subcaption}

\input epsf

 

\newcommand{\be}{\begin{equation}}
\newcommand{\ee}{\end{equation}}
\newcommand{\ba}{\begin{array}}
\newcommand{\ea}{\end{array}}
\newcommand{\bp}{\begin{proof}}
\newcommand{\ep}{\end{proof}}

\def\0{{\bf 0}}

\def\1{{\rm 1\hspace{-0.90ex}1} }
\newtheorem{theorem}{Theorem}

\newtheorem{definition}{Definition}

\newtheorem{proposition}{Proposition}

\newtheorem*{conclusion*}{Conclusion}

\renewcommand{\epsilon}{\varepsilon}
\def\twoplot[#1]#2#3#4#5{
\begin{figure}[hbt]
\begin{multicols}{2}
\begin{center}
    \includegraphics*[#1]{#2}
    \caption{\label{#2} #4}
\end{center}
\begin{center}
    \includegraphics*[#1]{#3}
    \caption{\label{#3} #5}
\end{center}
\end{multicols}
\end{figure}
}
\begin{document}
\bibliographystyle{plain}
\title[{IRREDUCIBILITY OF The THREE DIMENTIONAL aLBEVERIO- Rabanovich Representation OF THE PURE BRAID GROUP $p_3$}]
{IRREDUCIBILITY OF The THREE DIMENTIONAL aLBEVERIO- Rabanovich Representation OF THE PURE BRAID GROUP $P_3$}

\author[Hasan A. Haidar]{HASAN A. HAIDAR}
\author[MOHAMMAD N. ABDULRAHIM]{MOHAMMAD N. ABDULRAHIM}
\maketitle
\begin{abstract} We consider Albeverio- Rabanovich linear representation $\pi$ of the braid group $B_3$. After specializing the indeterminates used in defining the representation to non-zero complex numbers, we prove that the restriction of $\pi$ to the pure braid group $P_{3}$ of dimension three is irreducible.
\end{abstract}
\section{Introduction} Let $B_{n}$ be the braid group on $n$ strands. Consider the pure braid group $P_{n}$, the kernel of the obvious
surjective group homomorphism $B_{n}\rightarrow S_{n}$. Burau constructed a representations of $B_{n}$ of degrees $n$ and $n-1$, known as Burau and reduced Burau representations respectively \cite{Bu}. The reduced Burau representation of $B_{n}$ was proved to be irreducible \cite{F}. Using Burau unitarizable representation, Albeverio and Rabanovich presented a class of non trivial unitary representations for the braid groups $B_3$ and $B_4$ in the case where the dimension of the space is a multiple of $3$. 
Researchers gave a great value for representations of the pure braid group $P_n$. M.Abdulrahim gave a necessary and sufficient condition for the irreducibility of the complex specialization of the reduced Gassner representation of the pure braid group $P_n$ \cite{Ab}. In our work, we mainly consider the irreducibility criteria of Albeverio- Rabanovich representation of the pure braid group $P_3$ with dimension three. In section 3, we write explicitly Albeverio- Rabanovich representation $\pi$ of the braid group $B_3$ of dimension $(2n+m)\times(2n+m)$ . In section 4, we let $m=n=1$ and we write the images of the generators S and J of $B_3$ under $\pi$. Then we deduce the images of $\sigma_1$ and $\sigma_2$, the standard generators of $B_3$, under $\pi$. After that, we construct the representation $\phi$, the restriction of $\pi$ on the pure braid group $P_3$. In section 5, we prove that $\phi$ is an irreducible representation of $P_3$ of dimension three. 
\section{Preliminaries}
\begin{definition}
\cite{Bi} The braid group on n strings, $B_{n}$, is the abstract group with
presentation $B_{n}$ =$\{\sigma _{1},...,\sigma _{n-1};$ $\sigma _{i}\sigma
_{i+1}\sigma _{i}=\sigma _{i+1}\sigma _{i}\sigma _{i+1}$, for $%
i=1,2,...,n-2, $ $\sigma _{i}\sigma _{j}=$ $\sigma _{j}\sigma _{i}$ if $%
\left\vert i-j\right\vert > 1\}.$\newline
The generators $\sigma _{1},...,\sigma _{n-1}$ are called the standard
generators of $B_{n}$.\newline
\end{definition}

\vskip 0.45 in $\overline{\text{{\small Key words and}}}${\small \ phrases. Braid group, pure braid group, irreducible.}\\  {\small \ \ \ \ Mathematics Subject Classification Primary. 20F36}\\
\begin{definition}
\cite{Bi} The pure braid group, $P_{n}$, is defined as the kernel of the
homomorphism $B_{n}$ $\longrightarrow S_{n}$, defined by $\sigma
_{i}\longrightarrow $ $(i,i+1),1\leq i\leq n-1.$ It has the following
generators: 
\[
A_{ij}=\sigma _{j-1}\sigma _{j-2}...\sigma _{i+1}\sigma _{i}^{2}\sigma
_{i+1}^{-1}...\sigma _{j-2}^{-1}\sigma _{j-1}^{-1},1\leq i,j\leq n 
\]
\end{definition}
\begin{definition}
A representation $\gamma :G\longrightarrow GL(V)$ is said to be irreducible
if it has no non trivial proper invariant subspaces.
\end{definition}
\section{ Albeverio- Rabanovich representation of the braid group $B_3$}
Consider  the  braid  group $B_3$  and its standard generatores $\sigma _1$ and $\sigma_2$. $B_3$ will be  generated  by J and S and  has  only one  relation $S^2= J^3$, where  $J= \sigma _1 \sigma_2$  and $S=\sigma_1 J$.
Denote the representation of $B_3$  by  $\pi$, where $\pi (S)= U $  and $\pi(J)= V$. Here U and V are $(2n+m) \times (2n+m)$ block matrices given by\\
$$U = 2 \left( \begin{array}{ccc}
A-I_n/2 & B & C \\
B^\ast & B^\ast A^{-1} B-I_n/2 & B^\ast A^{-1} C \\
C^\ast & C^\ast A^{-1} B & C^\ast A^{-1} C-I_m/2
\end{array} \right)$$

and 

$$V= diag (I_n,\beta I_n,\beta ^2 I_m ).$$
We have $\beta =\sqrt[3]{1}$  is a primitive root, $ 1\leq m \leq n$, A and B are $n\times n$  matrices and C is an $n\times m$ matrix. We also have $V^3=I_{2n+m}$. If $A = A^\ast$ 
and $BB^\ast + CC^\ast= A - A^2$, we get $U = U^\ast$ and $U^2=I_{2n+m}$. For more details, see \cite{Al}.\\

\begin{proposition} \cite{Al} If A and B are invertible, $rank(C)=m$, $B^\ast B$ is a diagonal matrix with simple spectrum and every entry of A is non-zero then Albeverio- Rabanovich representation is irreducible. 
\end{proposition}
\section{Restriction of albeverio- Rabanovich representation $\pi$ on the pure braid group $P_3$}
Consider the braid group $B_3$ generated by S and V. Take $n = m = 1$ with B and C non zero real numbers and A is specialized to the value $\frac{1}{2}$. This implies that $B = B^\ast$, $C = C^\ast$ and $B^2 + C^2 = A - A^2$. For $A = \frac{1}{2}$, we have $B^2 = \frac{1}{4}- C^2$.
We require $-\frac{1}{2} < C < \frac{1}{2}$. We substitute  the value of B in U. So, we get\\

$$\pi(S) = U = 2 \left( \begin{array}{ccc}
0 & \sqrt{\frac{1}{4}-C^2} & C\\
\sqrt{\frac{1}{4}-C^2} & -2C^2 & 2C\sqrt{\frac{1}{4}-C^2}\\
 C & 2C\sqrt{\frac{1}{4}-C^2} & 2C^2-\frac{1}{2}
\end{array} \right)$$\\
 and 

  $$\pi(J) = V = \left( \begin{array}{ccc}
1 & 0 & 0\\
0 & \beta & 0  \\
0 & 0 & \beta^2
\end{array} \right)$$

\vskip .1in

Here, $\beta$ is a 3rd root of unity. That is , $\beta^3=1$.
\begin{proposition} The images of the standard generators $\sigma_1$ and $\sigma_2$ of the braid group $B_3$ under $\pi$ are given by:\\ 
$$\pi(\sigma_1) = \left( \begin{array}{ccc}
0 & \frac{2\sqrt{\frac{1}{4}-C^2}}{\beta} & \frac{2C}{\beta^2}\\
2\sqrt{\frac{1}{4}-C^2}& \frac{-4C^2}{\beta} & \frac{4C\sqrt{\frac{1}{4}-C^2}}{\beta^2}\\
2C & \frac{4C\sqrt{\frac{1}{4}-C^2}}{\beta} & \frac{4C^2-1}{\beta^2}
\end{array} \right)$$\\
and \\
$$\pi(\sigma_2) = \left( \begin{array}{ccc}
0 & 2\beta\sqrt{\frac{1}{4}-C^2} & 2\beta^2C\\
2\beta\sqrt{\frac{1}{4}-C^2} & -4\beta^2C^2 & 4C\sqrt{\frac{1}{4}-C^2}\\
2\beta^2C & 4C\sqrt{\frac{1}{4}-C^2} & \beta(4C^2-1)
\end{array} \right),$$\\
where C is non-zero real number such that $-\frac{1}{2} < C < \frac{1}{2}$, $\beta \neq 1$, and $\beta^3=1$.
\end{proposition} 
Now, apply Albeverio- Rabanovich representation $\pi$ on the pure braid group $P_3$. We get the following representation of dimension d = 3
\begin{proposition}
Let $\phi$ be the restriction of Albeverio- Rabanovich representation $\pi$ on the pure braid group $P_3$. Thus $\phi$ is defined as follows:\\

$$\phi(A_{12}) = \left( \begin{array}{ccc}
I & J & \beta P\\
\beta J & M & \beta^2Q\\
P & Q & R
\end{array} \right) \text{and}\ \phi(A_{23}) = \left( \begin{array}{ccc}
I & \beta^2J & \beta^2P\\
\beta^2J & M & \beta Q\\
\beta^2P & \beta Q & R
\end{array} \right),$$\\
where\\
$I=4\beta C^2(1-\beta)+\beta^2$, $J=8C^2(1-\beta)\sqrt{\frac{1}{4}-C^2}$, $P=2\beta C(1-\beta)(4C^2-1)$\\
$M=\beta^2(-4c^2+1)+16C^4(\beta-1)+4C^2$, $Q=4C(1-\beta)(4C^2+\beta^2)\sqrt{\frac{1}{4}-C^2}$,\\
$R=4\beta C^2-16C^4+4C^2+16\beta^2C^4-8\beta^2C^2+\beta^2$.\\ \\
Also,  $C$ is a non-zero real number such that $-\frac{1}{2} < C < \frac{1}{2}$, $\beta \neq 1$, and $\beta^3=1$.\\
As for $\phi(A_{13})=\sigma_2\sigma_1^2\sigma_2^{-1}$, we will not need it in our proof.
\end{proposition}
\section{Irreducibility of Albeverio- Rabanovich representation of the pure braid group $P_3$ of dimension three}
In this section, we prove that Albeverio- Rabanovich representation $\phi$ of the pure braid group $P_3$ of dimension three is irreducible.\\
\begin{theorem}
Albeverio- Rabanovich representation $\phi:P_{3}\longrightarrow GL_{3}(\mathbb{C})$ is irreducible
\end{theorem}
\begin{proof}
To get contradiction, suppose that this representation $\phi
:P_{3}\longrightarrow GL_{3}(C)$ is reducible .That is, there exists a
proper non-zero invariant subspace $S$, of dimension $1$ or $2$. It is clear that $\phi$ is unitary \cite{Al}. For a unitary representation, the orthogonal complement of a proper invariant subspace is again a proper invariant subspace. We then assume that $S$ is one dimensional subspace generated by a vector $v$. We use $e_{1},\ e_{2}$,\ and $e_{3}$ as the
canonical basis of $\mathbb{C}^{3}$. Let $\alpha$, $\alpha_1$ and $\alpha_2$ be non-zero complex numbers. It is easy to see that I,J, and Q are different from zero. We consider seven cases.
\begin{itemize}
 \item 
Case 1   Let $v=e_1$, it follows that $\beta\phi(A_{12})e_1-\phi(A_{23})e_1\in S$.\\
 Then $I(\beta-1)e_1+\beta P(1-\beta)e_3 \in S$. Hence $e_3 \in S$. \\
Also, we have $\beta^2\phi(A_{12})e_1-\phi(A_{23})e_1 \in S$. Then $I(\beta^2-1)e_1+J(1-\beta^2)e_2 \in S$.\\
 Hence $e_2 \in S$. Thus $S=\mathbb{C}^3$, a contradiction.\\
\item
Case 2   Let $v=e_2$, it follows that $\beta^2\phi(A_{12})e_2-\phi(A_{23})e_2\in S$.\\
Then $M(\beta^2-1)e_2+\beta Q(\beta-1)e_3 \in S$. Hence $e_3 \in S$.\\
Also, we have $\beta\phi(A_{12})e_2-\phi(A_{23})e_2 \in S$. Then $\beta J(1-\beta)e_1+M(\beta-1)e_2 \in S$.\\
Hence $e_1 \in S$. Thus $S=\mathbb{C}^3$, a contradiction.\\
\item
Case 3   Let $v=e_3$, it follows that $\beta\phi(A_{12})e_3-\phi(A_{23})e_3\in S$.\\
Then $Q(1-\beta)e_2+R(\beta-1)e_3 \in S$. Hence $e_2 \in S$.\\
Also, we have $\phi(A_{12})e_3-\beta\phi(A_{23})e_3 \in S$. Then $P(\beta-1)e_1+R(1-\beta)e_3 \in S$.\\
Hence $e_1 \in S$. Thus $S=\mathbb{C}^3$, a contradiction.\\
\item
Case 4   Let $v=e_1+\alpha e_2$, it follows that $\phi(A_{12})v=a_1v$ for some $a_1\in\mathbb{C}^*$. Then\\ $(I+\alpha J)e_1+(\beta J+\alpha M)e_2+(P+\alpha Q)e_3=a_1(e_1+\alpha e_2)$. 
Hence \begin{equation}\label{1} P+\alpha Q=0. 
\end{equation}
Also, there exists $a_2\in\mathbb{C}^*$ such that $\phi(A_{23})v=a_2v$. Then\\ $(I+\alpha\beta^2 J)e_1+(\beta^2 J+\alpha M)e_2+(\beta^2P+\alpha\beta Q)e_3=a_2(e_1+\alpha e_2)$. Hence \begin{equation}\label{2} \beta^2P+\alpha\beta Q=0. \end{equation}
Using (\ref{1}) and (\ref{2}) we get $P=0$, a contradiction.\\
\item
Case 5   Let $v=e_1+\alpha e_3$, it follows that $\phi(A_{12})v=a_1v$ for some $a_1\in\mathbb{C}^*$. Then \begin{equation}\label{3} J+\alpha \beta Q=0. \end{equation} Also, there exists $a_2\in\mathbb{C}^*$ such that $\phi(A_{23})v=a_2v$. Then \begin{equation}\label{4} \beta^2J+\alpha\beta Q=0. \end{equation} Using (\ref{3}) and (\ref{4}) we get $J=0$, a contradiction.\\
\item
Case 6   Let $v=e_2+\alpha e_3$, it follows that $\phi(A_{12})v=a_1v$ for some $a_1\in\mathbb{C}^*$. Then \begin{equation}\label{5} J+\alpha \beta P=0. \end{equation} Also, there exists $a_2\in\mathbb{C}^*$ such that $\phi(A_{23})v=a_2v$. Then \begin{equation}\label{6} \beta^2J+\alpha\beta^2 P=0. \end{equation} Using (\ref{5}) and (\ref{6}) we get $P=0$, a contradiction.\\
\item
Case 7   Let $v=e_1+\alpha_1 e_2+\alpha_2 e_3$, it follows that $\frac{1}{\beta-1}(\phi(A_{23})v-\phi(A_{12})v)=n_1v$ for some $n_1\in\mathbb{C}^*$. Then \\ \\ $\left(\begin{array}{c}
\alpha_1J(\beta+1)+\alpha_2\beta P\\
\beta J-\alpha_2\beta Q\\
P(\beta+1)+\alpha_1Q
\end{array}\right)=n_1\left(\begin{array}{c} 1\\ \alpha_1\\ \alpha_2\end{array}\right)$.\\
\vskip .1in
 Consider the following equations:
\begin{equation}\label{7} \alpha_1J(\beta+1)+\alpha_2\beta P=n_1 \end{equation}
\begin{equation}\label{8} \beta J-\alpha_2\beta Q=n_1\alpha_1 \end{equation}
\begin{equation}\label{9} P(\beta+1)+\alpha_1Q=n_1\alpha_2 \end{equation}
\\ Now, using $(\ref{7})$ and $(\ref{8})$, we get $\alpha_1^2J(\beta+1)+\alpha_1\alpha_2\beta P-\beta J+\alpha_2\beta Q=0$. 
Hence,\\ \begin{equation}\label{10} (\alpha_1P+Q)\alpha_2=\alpha_1^2J\beta+J. \end{equation}
Using (\ref{7}) and (\ref{10}) we get
\begin{equation}\label{11} (\alpha_1P+Q)n_1=-\beta^2J\alpha_1(\alpha_1P+Q)+\beta PJ(\alpha_1^2\beta+1). \end{equation}
 After multiplying (\ref{9}) by $(\alpha_1P+Q)^2$ and using (\ref{10}) and (\ref{11}), we get:
\begin{equation}\label{12}(P^2+J^2)Q\alpha_1^3+(-\beta^2P^2+2Q^2-\beta^2J^2)P\alpha_1^2+(-2\beta^2P^2+\beta^2J^2+Q^2)Q\alpha_1-\beta P(\beta Q^2+J^2)=0. \end{equation} \\
Similarly, there exists $n_2\in\mathbb{C}$ such that $\frac{1}{1-\beta^2}(\phi(A_{23})v-\beta^2\phi(A_{12})v)=n_2v$. Then \\ \\ $\left(\begin{array}{c}
I-\alpha_2 P\\
-J+\alpha_1M\\
\frac{\alpha_1Q\beta}{\beta+1}+\alpha_2R
\end{array}\right)=n_2\left(\begin{array}{c} 1\\ \alpha_1\\ \alpha_2\end{array}\right)$.\\
\vskip .1in
 Consider the following equations:\\
\begin{equation}\label{13} I-\alpha_2 P=n_2 \end{equation}
\begin{equation}\label{14} -J+\alpha_1M=n_2\alpha_1 \end{equation}
\begin{equation}\label{15} \frac{\alpha_1Q\beta}{\beta+1}+\alpha_2R=n_2\alpha_2 \end{equation}
\\The equations (\ref{13}) and (\ref{14}) give \begin{equation}\label{16} \alpha_1\alpha_2P=\alpha_1I+J-\alpha_1M \end{equation}
Now, using (\ref{14}), (\ref{15}), and  (\ref{16}), we get
\begin{equation}\label{17} -\beta^2PQ\alpha_1^3+(M-I)(M-R)\alpha_1^2+J(R-2M+I)\alpha_1+J^2=0. \end{equation}
Likewise, there exists $n_3\in\mathbb{C}$ such that $\frac{1}{1-\beta}(\phi(A_{23})v-\beta\phi(A_{12})v)=n_3v$. Then \\ \\ $\left(\begin{array}{c}
I-\beta\alpha_1 J\\
\alpha_1M-\alpha_2Q\\
-\beta P+\alpha_2R
\end{array}\right)=n_3\left(\begin{array}{c} 1\\ \alpha_1\\ \alpha_2\end{array}\right)$.\\ \\
 \vskip .1in
 Consider the following equations:
\begin{equation}\label{18} I-\beta\alpha_1 J=n_3 \end{equation}
\begin{equation}\label{19} \alpha_1M-\alpha_2Q=n_3\alpha_1 \end{equation}
\begin{equation}\label{20} -\beta P+\alpha_2R=n_3\alpha_2 \end{equation}
The equations (\ref{18}) and (\ref{19}) give
\begin{equation}\label{21} \alpha_2Q=\beta J\alpha_1^2+(M-I)\alpha_1. \end{equation}
Using (\ref{18}), (\ref{20}), and (\ref{21}) we get
\begin{equation}\label{22} \beta^2J^2\alpha_1^3+\beta J(R-2I+M)\alpha_1^2+(I-M)(I-R)\alpha_1-\beta PQ=0. \end{equation}

Now, we multiply (\ref{17}) by $J^2$, (\ref{22}) by $PQ$. After adding them, we get
\begin{equation}\label{23} a_1\alpha_1^2+b_1\alpha_1+c_1=0. \end{equation}

Also, we multiply (\ref{22}) by $Q(p^2+J^2)$, (\ref{12}) by $-\beta^2J^2$. After adding them, we get
\begin{equation}\label{24} a_2\alpha_1^2+b_2\alpha_1+c_2=0, \ where \end{equation}\\
$a_1=64\beta C^6(4C^2-1)^3(\beta-1)^4(4\beta C^2+\beta+2)$,\\
$b_1=128C^6(4C^2-1)^2\sqrt{\frac{1}{4}-C^2}(\beta-1)^4(8\beta C^2-\beta+4C^2-16C^4-1)$,\\
$c_1=-16C^4(4C^2-1)^2(\beta-1)^4(\beta-12\beta C^2+32\beta C^4-8C^2+32C^4-64C^6)$,\\
$a_2=128\beta C^7(4C^2-1)^3(\beta-1)^5(\beta-8\beta C^2+16\beta C^4+16C^4)$,\\
$b_2=256C^7\sqrt{\frac{1}{4}-C^2}(\beta^2+4C^2)(\beta-1)^5((4C^2-1)^2(16\beta C^4+4C^2+\beta-1)$,\\
$c_2=8\beta C^3(\beta-1)^3((4C^2-1)^2(768\beta C^8-192\beta C^6-16\beta C^4-8C^2+\beta+1)$.\\\\
The equations (\ref{23}) and (\ref{24}) give 
\begin{equation}\label{25} \alpha_1=\frac{K}{8C^2\sqrt{\frac{1}{4}-C^2}(-4C^2+3\beta+2)}, \ where \end{equation}\\
$K=-\beta+12\beta C^2-32\beta C^4+8C^2-32C^4+64C^6.$\\\\
Using (\ref{21}) and (\ref{25}) we get\\\\
\begin{equation}\label{26} \alpha_2=\frac{-K(\beta+8\beta C^2-48\beta C^4+64\beta C^6-4C^2-16C^4+64C^6+1)}{8C^3(4C^2-1)(\beta^2+4C^2)(-4C^2+3\beta+2)^2}. \end{equation}
\\Now, using (\ref{7}), (\ref{25}), and (\ref{26}) we obtain:\\
\begin{equation}\label{27} n_1=\frac{-K(\beta-1)(\beta+32\beta C^4+16C^4-64C^6)}{4C^2(\beta^2+4C^2)(-4C^2+3\beta+2)^2}. \end{equation}
After substituting the equations (\ref{25}), (\ref{26}), and (\ref{27}) in (\ref{8}), we obtain equation (28) \\\\
$\beta(256C^4-28C^2-128C^6-12032C^8+75776C^{10}-94208C^{12}-229376C^{14}+196608C^{16})+304C^4-2432C^6+5632C^8+41984C^{10}-208896C^{12}+98304C^{14}+196608C^{16}-1=0$.\\\\
We substitute $\beta=\frac{-1}{2}\pm \frac{\sqrt{3}}{2}i$. We then take the imaginary part and real part in (28) to be zero. Thus, we obtain the equations (29) and (30) respectively\\\\
$4C^2(4C^2-1)(12288C^{12}-11264C^{10}-8704C^8+2560C^6-112C^4-36C^2+7)=0$.\\\\
$98304C^{16}+212992C^{14}-161792C^{12}+4096C^{10}+11648C^8-2368C^6+176C^4+14C^2-1=0$. \\\\
Equation (29) has 10 rejected solutions and two accepted solutions which are $\pm 0.43$ (rounded to the nearest hundredth).\\
Equation (30) has 14 rejected solutions and two accepted solutions which are $\pm 0.23$ (rounded to the nearest hundredth).\\
Thus we have no common solutions between (29) and (30), a contradiction.
\end{itemize}
\end{proof}

\bigskip
\footnotesize 
Hasan A. Haidar, \textsc{Department of Mathematics and Computer Science, Beirut Arab University, P.O. Box 11-5020, Beirut, Lebanon}\par\nopagebreak
\textit{E-mail address}: \texttt{hah339@student.bau.edu.lb}
 
Mohammad N. Abdulrahim, \textsc{Department of Mathematics and Computer Science, Beirut Arab University, P.O. Box 11-5020, Beirut, Lebanon}\par\nopagebreak
\textit{E-mail address}: \texttt{mna@bau.edu.lb}
\end{document}